\documentclass[a4paper, 11pt]{amsart}
\usepackage{geometry}
\usepackage{amsmath,amssymb,amsfonts,amsthm} 
\usepackage{color}
\usepackage{fullpage}
\usepackage[all]{xy}
\renewcommand{\P}{\mathbb P}
\DeclareMathOperator{\Gr}{Gr}
\DeclareMathOperator{\Proj}{Proj}
\DeclareMathOperator{\Sing}{Sing}
\newtheorem{thm}{Theorem}
  \newtheorem{cor}[thm]{Corollary}
\newtheorem{lem}[thm]{Lemma}
\newtheorem{prop}[thm]{Proposition}

\theoremstyle{definition}
\newtheorem{prob}[thm]{Problem}

\theoremstyle{remark}
\newtheorem{rem}[thm]{Remark}
\theoremstyle{remark}

\numberwithin{equation}{section}
\numberwithin{thm}{section}
\newcommand{\PP}{{\mathbb{P}}}
\newcommand{\FF}{{\mathbb{F}}}

\newcommand{\C}{{\mathbb{C}}}
\newcommand{\Ce}{{\mathcal{C}}}
\newcommand{\ZZ}{{\mathbb{Z}}}
\newcommand{\oo}{{\mathcal O}}
\newcommand{\fie}{\varphi}
\newcommand{\Xhat}{\hat X}
\newcommand{\Yhat}{\hat Y}
\newcommand{\Xihat}{\hat \Xi}
\newcommand{\Thetahat}{\hat \Theta}
\DeclareMathOperator{\Pf}{Pf}
\DeclareMathOperator{\codim}{codim}
\DeclareMathOperator{\Hom}{Hom}
\DeclareMathOperator{\Pic}{Pic}
\DeclareMathOperator{\Cl}{Cl}
\DeclareMathOperator{\Sec}{Sec}
\DeclareMathOperator{\Aut}{Aut}
\DeclareMathOperator{\Exc}{Exc}
\title[Arithmetically Gorenstein Calabi--Yau threefolds in $\PP^7$]{Arithmetically Gorenstein Calabi--Yau\\ threefolds in $\PP^7$}
\date{\today}
\author[S. Coughlan]{Stephen Coughlan}
\address[SC]{Institut f\"ur Algebraische Geometrie, Leibniz Universit\"at Hannover, Welfengarten 1, D-30167 Hannover, Germany}
\curraddr{Mathematisches Institut, Lehrstuhl Mathematik VIII, Universit\"at Bayreuth,
Universit\"atsstrasse 30, D-95447 Bayreuth, Germany}
\email{stephencoughlan21@googlemail.com}
\author[{\L}. Go{\l}{\c e}biowski]{{\L}ukasz Go{\l}{\c e}biowski}
\address[{\L}G]{Jagiellonian University Cracow}
\email{l.turbo.golebiowski@gmail.com}
\author[G.~Kapustka]{Grzegorz Kapustka}
\address[GK]{Jagiellonian University Cracow, University of Z\"urich}
\email{grzegorz.kapustka@uj.edu.pl}
\author[M.~Kapustka]{Micha{\l} Kapustka}
\address[MK]{University of Stavanger, Department of Mathematics and Natural Sciences, NO-4036 Stavanger, Norway}
\email{michal.kapustka@uis.no}

\begin{document}

\begin{abstract} 
We present a list of arithmetically Gorenstein Calabi--Yau threefolds in $\PP^7$ and give evidence that this is a complete list.
In particular we construct three new families of arithmetically Gorenstein Calabi--Yau threefolds in $\PP^7$ for which no mirror construction is known.
\end{abstract}

\maketitle


\section{Introduction}
By a \emph{Calabi--Yau} threefold we mean a smooth three dimensional compact complex projective manifold $X$ with $K_X=0$ and $h^1(\oo_X)=0$. 
  It is known that Calabi--Yau threefolds embedded in $\PP^5$ are complete intersections; either of two cubics, or of a quadric and a quartic, or of a quintic and a hyperplane. In \cite{KK1} we studied Calabi--Yau threefolds embedded in $\PP^6$ giving the complete classification of the examples of degree $d\leq 14$ (see also \cite{Ton}).
  Note that the degree of such threefolds in $\PP^6$ is bounded by $41$, but the highest known degree is $17$ (see \cite{KK2}).
  It is an interesting problem to find the highest degree that a Calabi--Yau threefold in $\PP^6$ can have.
  Recall also that all smooth threefolds (so also all Calabi--Yau threefolds) can be smoothly projected to $\PP^7$. In particular all Calabi--Yau threefolds can be embedded in $\PP^7$.
  
  The aim of this article is to study Calabi--Yau threefolds of codimension $4$ in $\PP^7$ that are projectively normal. More precisely we study Calabi--Yau threefolds $X\subset \PP^7$ such that the homogeneous coordinate ring $S(X) = \C[x_0, \dots , x_7]/ I_X$ (where $I_X$ is the saturated ideal of $X$) is a Cohen--Macaulay ring. This is equivalent to saying that $h^i(\mathcal{I}_X(j))=0$ for $0<i<\codim X$. A priori, arithmetically Cohen--Macaulay is a weaker condition than arithmetically Gorenstein, but for Calabi--Yau threefolds, the two are equivalent. Thus we consider arithmetically Gorenstein (aG) Calabi--Yau threefolds in $\PP^7$.
  
Obtained examples extend the lists of known Calabi--Yau threefolds described for example in \cite{IIM}, \cite{GKM}, \cite{K2}, \cite{C}, \cite{CCM}, \cite{BKZ}, \cite{QS}. Furthermore, the considered Calabi--Yau threefolds have simple algebraic descriptions and in this way are natural test examples for the mirror symmetry conjecture (cf.~\cite{BP}, \cite{JKLM}). 
 
  Another motivation for studying codimension 4 aG submanifolds is the search for a general structure theorem.
 We say that a submanifold $X \subset \PP^n$ is subcanonical if $K_X = \oo_X (k)$
for some $k \in \ZZ$. 
By the so-called Serre construction, a codimension $2$ subcanonical submanifold is always the zero locus of a section of a rank two vector bundle $E$. If it is additionally aG then then it is a complete intersection. 
A codimension $3$ subcanonical manifold $X$ is Pfaffian see \cite{O2}, \cite{W}, i.e.
it is the first nonzero 
degeneracy locus of a skew-symmetric morphism of vector bundles of odd rank $E (-t) \to E$ for some $t \in \ZZ$. In this case, $X$ is given locally by the vanishing of $2u\times 2u$ Pfaffians of an alternating map from the vector bundle $E$ of odd rank $2u + 1$ to its twisted dual such that we have the following resolution:
$$0\to \oo_{\PP^n}(-2s-t)\to  E (-s-t)\to E(-s)\to \mathcal{I}_X \to 0 ,$$
where $s = c_1(E) + ut$. If a subcanonical codimension $3$ manifold is additionally aG, then $E$ can be chosen to be a sum of line bundles.
Note that Tonoli in his PhD thesis \cite{Ton} constructed examples of Calabi--Yau threefolds in $\PP^6$ using the Pfaffian resolution; see also \cite{KK3}.

Recently, Reid \cite{R} described the shape of a projective resolution for a codimension 4 Gorenstein ideal.
More precisely, if $X$ is an aG manifold of codimension 4 then the ideal $\mathcal{I}_X$ admits a free resolution of the form:
\begin{equation} \label{resolution4}
0\rightarrow P_4\rightarrow P_3\xrightarrow{M'} P_2\xrightarrow{M} P_1\xrightarrow{L} \mathcal{I}_X\rightarrow 0,
\end{equation}
where $P_4=\oo(-k)$ for some integer $k$, $P_i\cong\Hom(P_{4-i},P_4)$, $M=(A,B)$ is a $(k+1)\times 2k$ matrix with polynomial entries  made of two blocks $A,B$ 
satisfying  $A(B^t)+B(A^t)=0$ and $M'=\left(\begin{smallmatrix}B^t\\A^t\end{smallmatrix}\right)$.
Conversely, if $M=(A,B)$ is a matrix as above for which the rank $<k$ locus $D_{k-1}$ satisfies $\operatorname{codim} D_{k-1}\geq 4$,
then there exists an arithmetically Gorenstein $X$ of codimension 4 with resolution  (\ref{resolution4}).
As Reid observes, it is difficult to find matrices $M$ a priori satisfying the above rank condition. It is even an interesting problem to write out $M$ a posteriori for some of the examples in our list.

In Table 1 we present a list of examples of aG Calabi--Yau threefolds in $\PP^7$, that we conjecture
to be a complete list. 
We hope that this list provides working examples that can be used to further develop the analysis from \cite{R} of the resolution of the ideal of $X$ and the matrices $A$ and $B$. 
The methods of construction are: bilinkages \cite{PS}, unprojections \cite{Reid Kinosaki}, \cite{PR} and the Serre construction in scrolls.
We complete the list of known examples presented in \cite{B}, \cite{BG}, \cite{Ku0}, \cite{Ku}, \cite{GK-jl}.
Moreover, we study the birational models of our examples.
Finally, we announce partial classification results for aG Calabi--Yau threefolds in $\PP^7$.

We have used computer algebra packages Magma \cite{Magma} and Macaulay2 \cite{M2} in various places, which are highlighted in the text.


{\bf Acknowledgements.} We thank Ch.~Okonek for useful discussions and advice. We thank S.~Cynk, S.~Galkin, A.~Kresch and K.~Ranestad for useful comments.  S. Coughlan was partially supported by the DFG through
grant Hu 337-6/2. {\L}.~Go{\l}{\c e}biowski, G.~Kapustka, M.~Kapustka are supported by NCN grant 2013/10/E/ST1/00688. Part of this work was done during the Polish Algebraic Geometry mini-Semester (miniPAGES), which was supported by the grant 346300 for IMPAN from the Simons Foundation and the matching 2015--2019 Polish MNiSW fund. 
\section{Preliminaries}
\subsection{Degree bound}
Let $X$ be a nonsingular aG Calabi--Yau 3-fold in $\PP^7$.
\begin{lem} The degree of $X$ takes values between $14$ and $20$.
\end{lem}
\begin{proof}
Kodaira vanishing and the Riemann--Roch formula for $X$ polarised by $A$ gives 
$$h^0(mA)=\frac{1}{6} (mA)^3+\frac{1}{12} mA \cdot c_2.$$
Since $X\subset\P^7$ is projectively normal and $h^0(A)=8$, we get the following relationship between Chern classes of $X$
\[2A^3+A\cdot c_2=96.\]
Again by projective normality, $h^0(2A)\leq h^0(\mathcal{O}_{\PP^7}(2))=\left(\begin{smallmatrix}7+2\\ 7\end{smallmatrix}\right)=36$, so we get $A^3\leq 20$. The lower bound for $A^3$ is $14$, coming from the Castelnuovo inequality for surfaces of general type with birational canonical map.
\end{proof}
\subsection{Bilinkage}
Our basic tool of construction will be bilinkage.
The fundamental paper of Peskine and Szpiro \cite{PS} introduced modern foundations for the theory of linkage:
Let  $U$, $V$ be reduced schemes with no common components, and suppose that $U$ and $V$ are 
contained in an arithmetically Gorenstein scheme $G\subset\PP^n$. We say that $U$ and $V$ are 
\emph{directly linked} if their union is a complete intersection $F$ in $G$ (for general subschemes $F$, the 
correct definition is that $\mathcal{I}_F : \mathcal{I}_U =\mathcal{I}_V$ and $\mathcal{I}_F :\mathcal{I}_V=
\mathcal{I}_U$). Sometimes we say that $U$ is linked to $V$ through $G$. This defines an equivalence 
relation, which we call \emph{Gorenstein linkage}. We say that two reduced schemes $U$ and $W$ are 
\emph{bilinked} if they can be linked in two steps.

\subsection{Hodge numbers}

Bertin \cite{B} showed how to compute the Hodge numbers of some Calabi--Yau threefolds, and this method can be applied to some of the families in Table 1. Another way to find Hodge numbers appears in \cite{DFF}, and this works by computing the graded pieces of the deformation space $T^1$. Unfortunately the complexity of these computations becomes very high when the degree is greater than $17$. The examples obtained by bilinkage can also be constructed using unprojection, in which case we compute the Hodge numbers using \cite{GK-jl}. All the above methods rely on an explicit understanding of the defining ideal of the variety. In some cases, we can also compute Hodge numbers directly from the construction.

\section{The table}
In Table 1 we give possibly the complete list of aG Calabi--Yau threefolds in $\PP^7$. In the table and elsewhere, $\Pf_{13} \subset \PP^7$ is a codimension $3$ manifold defined by a
$5 \times 5$ skew-symmetric matrix with one row of quadrics, $F_1$ is the del Pezzo threefold
$(1, 1) \subset \PP^2 \times \PP^2$, and $F_2$ is the del Pezzo threefold $\PP^1 \times \PP^1 \times \PP^1$.

\renewcommand{\arraystretch}{1.3}\begin{table}[!ht] 
\begin{tabular}{ccccc}
\hline
No. & Deg. & $h^{1,1}$ & $h^{1,2}$ & Description \\
\hline
1 & 14 & 2 & 86 & $(2,4)$ divisor in $\PP^1\times \PP^3$ \\
2 & 15 & 1 & 76 & $G(2,5)\cap X_3\cap \PP^7$ \\
3 & 16 & 1 & 65 & $X_{2,2,2,2}$ \\
4 & 17 & 1 & 55 &  bilinked on $Y_{2,2,2}$ to $\PP^3$ \\
5 & 17 & 2 & 58 & $2\times 2$ minors of a $3\times 3$ matrix with degrees $\left(\begin{smallmatrix}1 & 1 & 1 \\ 1 & 1 & 1 \\ 2 & 2 & 2 \end{smallmatrix}\right)$ \\
6 & 17 & 2 & 54 & rolling factors, codim 2 in cubic scroll  \\
7 & 18 & 1 & 46 & bilinked on $Y_{2,2,3}\subset \PP^7$ to $F_1$ 

\\
8 & 18 & 1 & 45& bilinked on $Y_{2,2,3}\subset \PP^7$ to $F_2$ 
\\ 
9& 19 & 2 & 37 & bilinked on special $\Pf_{13}$ to $F_1$ \\
10& 19 & 2 & 36 & bilinked on special $\Pf_{13}$ to $F_2$ \\

11 & 20 & 2 & 34 & $3\times 3$ minors of $4\times 4$ matrix with linear forms in $\PP^7$\\
\hline
\end{tabular}
\caption{aG Calabi--Yau threefolds in $\PP^7$}
\end{table}
Examples $2$, $4$, $5$ and $11$ were studied in \cite{B}, examples $4$, $7$ and $8$ in \cite{GK-jl} and $1$ in \cite{Ku}.
Examples $6$, $9$ and $10$ are new families of Calabi--Yau threefolds. 
We will give a more precise description of each family in section \ref{Ex}.
We study the geometry of these examples, describing their birational models and their K\"ahler cones (in \cite{CGKK} we shall complete this analysis). In order to study the examples of degree $19$ we use the Cremona transform defined by quadrics containing the Segre embedding of $\PP^2\times \PP^2$ in $\PP^8$.
Our main result in \cite{CGKK} is that Table 1 gives the complete list of aG Calabi--Yau threefolds in $\PP^7$ of degree $\leq 17$.
We have evidence that our Table is complete for degrees $18$, $19$ and $20$. 
\begin{prob} Does Table 1 give the complete classification of aG Calabi--Yau threefolds in $\PP^7$?
\end{prob}

\section{Examples}\label{Ex}
In this section we discuss case by case the examples described in Table 1.

\subsection{Degree 14} 
The anti-canonical divisor of $X_{14}\subset \PP^1\times \PP^3$ is a Calabi--Yau threefold of degree $14$.
To see that $X_{14}\subset \PP^7$ is aG, we use the long cohomology exact sequence associated to
$$0\to \mathcal{I}_X\to \mathcal{I}_{\PP^1\times \PP^3} \to\mathcal{I}_{X | \PP^1\times\PP^3}\to 0,$$
and the fact that $\mathcal{I}_{X | \PP^1\times\PP^3}=\oo_{\PP^1\times\PP^3}(-2,-4)$.
Note that this threefold was considered in \cite[Thm.~5.4]{Ku} and we deduce the following:
\begin{prop}\label{deg14} The Picard number of $X_{14}$ is $2$ i.e.~$h^{1,1}=2$ and $h^{1,2}=86$.
One extremal contraction is a K3 fibration whereas the other is a contraction of $64$ rational curves.
The minimal resolution of the ideal defining $X_{14}\subset \PP^7$ is
\[
0\to \oo(-8)\to 6\oo(-6)\oplus 3\oo(-4)\to 8\oo(-5)\oplus 8\oo(-3)\to3\oo(-4)\oplus 6\oo(-2)\to \mathcal{I}_{X_1}\to 0.
\]
\end{prop}
\begin{proof} From the Lefschetz hyperplane theorem, we know that $h^{1,1}(X_{14})=2$ because $\Pic(\PP^1\times\PP^3)=\ZZ\otimes\ZZ$. Since $X_{14}$ is a hypersurface in a toric
variety (cf.~Prop.~\ref{prop!toric-hodge}), we can compute the Euler number $e(X_{14})=-168$. Then since $e(X)=2(h^{1,1}-h^{1,2})$ for a Calabi--Yau threefold, we get $h^{1,2}(X)=86$.
From \cite[Theorem 4.9]{Ku}, $X_{14}$ contains exactly $64$ rational curves that are contracted by the projection to $\PP^3$. The projection to $\PP^1$ gives a K3 fibration. 
\end{proof}
\subsection{Degree $15$}
Let $Y$ denote a linear $\PP^7$ section of the Pl{\"u}cker embedding of the Grassmannian $\Gr(2,5)$ in $\PP^9$. The Lefschetz hyperplane theorem implies that the Picard group of $Y$ is generated by the hyperplane section $H$ on $\PP^7$, and moreover, the canonical bundle of $Y$ is $-3H$. It follows that the intersection of $Y$ with a cubic is a Calabi--Yau threefold. This classical example has Picard number $1$ and is studied more precisely in \cite{B}. They are aG since the Pl\"ucker embedding of the Grassmannian $\Gr(2,5) $ is aG.
 
\subsection{Degree $16$}
This is the intersection of four quadrics in $\PP^7$. Smoothness follows from the Bertini theorem and the canonical divisor is computed from the adjunction formula.
This example is aG from the cohomology exact sequence associated to the restriction exact sequence.

\subsection{Degree $17$ no.~$4$}\label{deg17no4}
Let $\Pi\cong\PP^3$ and consider a linear embedding of $\Pi$ in $\PP^7$.
We can perform a bilinkage of $\Pi$ inside the complete intersection $Y=Y_{2,2,2}\subset \PP^7$ of three quadrics containing $\Pi$ by first linking through a linear section and then through a cubic. Since $\Pi\subset \PP^7$ is aG and bilinkage preserves this property, we obtain in this way an aG threefold $X$ of degree $17$ in $\PP^7$.
Furthermore, from the mapping cone construction we can find a free resolution of $I_X$ and in consequence compute $K_X$.

We proceed as follows: the ideal defining the complete intersection $S_{2,2,2,1}$ has the Koszul resolution
\[
\mathcal{S}\colon0\rightarrow\oo(-7)\rightarrow 3\oo(-5)\oplus\oo(-6)\rightarrow3\oo(-3)\oplus 3\oo(-4)\rightarrow\oo(-1)\oplus 3\oo(-2)\rightarrow\mathcal{I}_{S_{2,2,2,1}}\rightarrow 0.
\]
On the other hand, the Koszul resolution of the ideal defining $\Pi$ is
\[\mathcal{F}\colon0\rightarrow\oo(-4)\rightarrow 4\oo(-3)\rightarrow 6\oo(-2)\rightarrow 4\oo(-1)\rightarrow \mathcal{I}_{\Pi}\rightarrow 0.\]

The inclusion of ideals $I_S\subset I_{\Pi}$ induces a map $\alpha\colon\mathcal{S}\to\mathcal{F}$, and a slightly modified version of the mapping cone of $\alpha^{\vee}(-7)$ gives a resolution of the ideal defining the linked variety $G$ (cf.~\cite[Prop.~1.10]{No}):
\begin{align*}
\mathcal{G}\colon0\rightarrow4\oo(-6)\rightarrow \oo(-6)\oplus 9\oo(-5)\rightarrow 7\oo(-4)&\oplus3\oo(-3)\rightarrow \\
& \oo(-3)\oplus 3\oo(-2)\oplus\oo(-1)\rightarrow \mathcal{I}_{G}\rightarrow 0.
\end{align*}
Now, we can take the cubic in $\mathbb{P}^7$ which contains $G$ and perform a second linkage. Note that for the complete intersection $T_{2,2,2,3}$, the Koszul complex is
\[
\mathcal{T}\colon0\rightarrow\oo(-9)\rightarrow 3\oo(-7)\oplus\oo(-6)\rightarrow 3\oo(-4)\oplus3\oo(-5)\rightarrow 3\oo(-2)\oplus \oo(-3)\rightarrow \mathcal{I}_{T_{2,2,2,3}}\rightarrow 0.
\]
Then the mapping cone of $\beta^\vee(-9)$ for $\beta\colon\mathcal{T}\to\mathcal{G}$ gives the free resolution of $\mathcal{I}_X$:
\begin{align*}
\mathcal{X}'\colon0\rightarrow\oo(-8) \oplus 3\oo(-7)&\oplus\oo(-6)\rightarrow 3\oo(-7)\oplus4\oo(-6)\oplus 7\oo(-5)\rightarrow \\
& 3\oo(-5)\oplus 12\oo(-4)\oplus\oo(-3)\rightarrow 5\oo(-3)\oplus 3\oo(-2)\rightarrow\mathcal{I}_{X}\rightarrow 0.
\end{align*}

To get a minimal resolution we need to quotient out terms corresponding to maps of degree 0. Since the three quadrics and the cubic defining $T$ form part of a minimal set of generators for the ideal of $G$, we can replace the last term $\oo(-8)\oplus 3\oo(-7)\oplus2\oo(-6)$ in $\mathcal{X'}$ by the dual to the cokernel of the inclusion $[3\oo(-2)\oplus\oo(-3)\to\mathcal{G}_1]$ twisted by $\oo(-9)$, which is simply $\oo(-8)$ (see \cite[Exercise 21.23]{E}). Moreover, since the quadrics in the ideal of $X$ are a subset of generators of the complete intersection $T$, they do not admit any linear syzygy. It follows that the syzygy $\oo(-3)\rightarrow 5\oo(-3)\oplus 3\oo(-2)$ eliminates a cubic generator of $\mathcal{I}_X$, and hence the corresponding part can be removed from the resolution. We finally remove the $3\oo(-5)$ from the central term, because $X$ is aG so the minimal resolution must be symmetric. In the end we get the minimal resolution:
\begin{align*}
\mathcal{X}\colon0\rightarrow\oo(-8) \rightarrow 3\oo(-6)\oplus  4\oo(-5)\rightarrow &12\oo(-4)\rightarrow3\oo(-2)\oplus4\oo(-3)\rightarrow\mathcal{I}_{X}\rightarrow 0.
\end{align*}
In particular, since $K_{X}=\mathrm{Ext}^4_R(R/I_{X},R)(-8)$ and $I_{X}=\mathrm{Ann}_R(K_{X})$, we deduce that $K_{X}\cong\oo_{X}$.
 From the resolution we obtain also $h^1(X,\oo_X)=0$. Finally, smoothness of $X$ follows from computer calculation of one example.  
\begin{rem} The above example together with its minimal free resolution was considered in \cite{B}. It is a Calabi--Yau threefold of degree 17 in $\PP^7$ with resolution of Kustin--Miller type. It was also constructed using unprojection of a del Pezzo of degree $5$ in a complete intersection $2,2,3$ in \cite{GK-jl}.
\end{rem}
\begin{lem} The generic element of the deformation family of $X$ is obtained via the bilinkage construction as above.
\end{lem}
\begin{proof}
We make a parameter count to show that the dimension of the obtained family (of Calabi-Yau threefolds up to isomorphism) is equal to $h^{1,2}(X)=55$, which was computed in \cite{B}. Indeed, the dimension of the family of complete intersections $Y_{2,2,2}$ containing $\Pi$ is
\[\dim\Gr(3,H^0(\mathcal{I}_{\Pi}(2)))-\dim\Aut(\PP^7,\Pi)=69-47=22.\]
Here we compute from the above resolution $\mathcal{F}$ that the space of quadrics containing $\Pi$ is 26-dimensional, and $\Aut(\PP^7,\Pi)$ is the space of automorphisms of $\PP^7$ preserving $\Pi$. Since the bilinkage is of degree 2 (see \cite{Hartshorne-biliaison} for that interpretation of bilinkage), we get an additional $h^0(\oo_Y(\Pi+2H))=1+h^0(Y,2H)=33$ parameters. Thus we have in total $55$ parameters, which implies that a generic element of the deformation family of $X$ is obtained via our construction.
\end{proof}
From the above resolution $\mathcal{X}$  we deduce the following more explicit description of $X$:
\begin{prop}\label{prop!deg-17-ci}
Let $M$ be a $4\times3$ matrix, $v$ a $4\times1$ column vector, each with linear entries in $x_1,\dots,x_8$.
The threefold defined by equations
\[Mv=\bigwedge^3M=0\]
is an aG Calabi--Yau threefold in $\PP^7$ of degree 17 as described in Table 1 no.~4.
\end{prop}
\begin{proof} We perform the above bilinkage backwards and after passing to the minimal resolution we end up with a Koszul resolution of a complete intersection of four hyperplanes.
\end{proof}

\subsection{Degree $17$ no.~$5$}
We work in the language of toric varieties as described in 
\cite{BZ}, \cite{BCZ}, 
\cite{A}.
Let $\Ce$ be the toric variety $\PP_{\PP^1}(2\oo\oplus3\oo(-1))$ with weight matrix
\[\Ce=
\begin{array}{cccccccc}
& t_1 & t_2 & a_1 & a_2 & b_1 & b_2 & b_3 \\
\hline
h & 1 & 1 & 0 & 0 & -1 & -1 & -1 \\
\eta & 0 & 0 & 1 & 1 & 1 & 1 & 1 
\end{array}
\]
and irrelevant ideal $(t_1,t_2)\cup(a_1,a_2,b_1,b_2,b_3)$.
Then take $\Xhat\subset\Ce$ to be the vanishing locus of the minors of the $2\times3$ matrix with entries
\[\Ce\supset\Xhat\colon\bigwedge^2\begin{pmatrix}q_1&q_2&q_3\\b_1&b_2&b_3\end{pmatrix}=0\]
where $q_i$ are general elements of the linear system $|2\eta|$. 

The toric minimal model program for $\Ce$ runs as follows
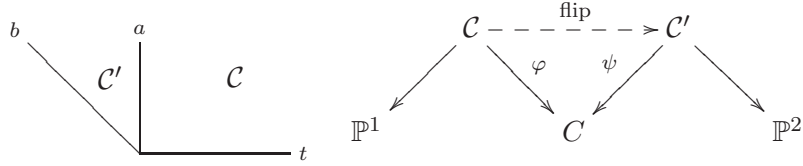
\[\begin{array}{lr}
{\xygraph{
!{(0,-1.3) }="A"
!{(0,0) }*+{\scriptstyle{a}}="B"
!{(1.7,-1.3) }*+{\scriptstyle{t}}="C"
!{(-1.3,0) }*+{\scriptstyle{b}}="D"
!{(1,-0.5) }*+{\Ce}
!{(-0.3,-0.5) }*+{\Ce'}
"A"-"B"  "A"-"C" "A"-"D"
}} &
{\xymatrix{
&\Ce\ar[ld]\ar[rd]^{\fie}\ar^{\textrm{flip}}@{-->}[rr]
&&{\Ce'}\ar[ld]_{\psi}\ar[rd]\\
\PP^1&&C&&\PP^2
}}
\end{array}\]
The rays of the secondary fan on the left, match up with the diagram of maps on the right. It is a standard computation (see below for an example) to work out all the maps in coordinates on $\Ce$ and its counterpart toric variety $\Ce'=\PP_{\PP^2}(2\oo\oplus2\oo(-1))$, with weight matrix
\[\Ce'=
\begin{array}{cccccccc}
& t_1 & t_2 & a_1 & a_2 & b_1 & b_2 & b_3 \\
\hline
h' & 1 & 1 & 1 & 1 & 0 & 0 & 0 \\
\eta' & -1 & -1 & 0 & 0 & 1 & 1 & 1 
\end{array}\]
and irrelevant ideal $(t_1,t_2,a_1,a_2)\cup(b_1,b_2,b_3)$. Note that the birational transform of $\eta$ (respectively $h$) under $\Ce\dashrightarrow\Ce'$ is $h'$ (resp.~$h'-\eta'$).

Let $\fie\colon\Ce\to C\subset\PP^7$ be the map induced by $|\eta|$, given in coordinates by $x_1=t_1b_1,x_2=t_1b_2,x_3=t_1b_3,x_4=t_2b_1,x_5=t_2b_2,x_6=t_2b_3,x_7=a_1,x_8=a_2$, so that $C$ is a five-dimensional cone over $\PP^1\times\PP^2$ with vertex $\Pi\cong\PP^1$, and $\fie$ is the blowup of $\Pi$ with exceptional locus $\PP^1\times\PP^2$. We find that $X=\fie(\Xhat)$ is defined by the $2\times2$ minors of the $3\times3$ matrix
\[\PP^7\supset X\colon\bigwedge^2\begin{pmatrix}
q_1 & q_2 & q_3 \\ x_1 & x_2 & x_3 \\ x_4 & x_5 & x_6
\end{pmatrix}=0\]
as described in Table 1 (section of weighted Segre embedding of $\PP^2\times\PP^2$).

Now $X$ has an elliptic fibration to $\PP^2$ defined by the ratios between the columns of the above matrix. 
Indeed, let $\Xhat'$ in $\Ce'$ denote the birational image of $\Xhat$. Over each fibre of $\Xhat'\to\PP^2$, the three defining quadrics have a single linear dependency induced by the syzygy, and thus the fibre is an elliptic curve. It remains to check directly that $\Xhat'$ and $X$ are isomorphic, because $\Xhat'$ is disjoint from the exceptional locus of $\psi\colon\Ce'\to C$.

The restriction of the flip $\Ce\dasharrow\Ce'$ to $\Xhat$ simply contracts a $\PP^1$-bundle over $\Pi\subset X$, and this resolves the indeterminacy of the rational map $X\dasharrow\PP^1$ to give a K3 fibration on $\Xhat$. We can deduce the following (the proof will be given in \cite{CGKK}):
 \begin{cor}\label{mov(X1)} Calabi--Yau threefold no.~$5$ has Picard number 2, and the K\"{a}hler cone is generated by the divisors $H_1$ and $H_2$, such that $|H_1|$ induces an elliptic fibration and $|H_2|$ contracts a line $\Pi$. After flopping $\Pi$, we obtain another birational model $X'$ admitting a K3 fibration. 
The K\"{a}hler cones of these two threefolds cover the movable cone $Mov(X)$ such that there are no more birational models of $X$ (cf. \cite{Ka}).  
\end{cor}
\begin{rem} This example is discussed from the point of view of mirror symmetry in \cite{JKLM}.
\end{rem}

\subsection{Degree $17$ no.~6}
We start from the toric 6-fold $\FF$ with weight matrix
\[\FF=
\begin{array}{ccccccccc}
& t_1 & t_2 & a_1 & a_2 & b_1 & b_2 & b_3 & c \\
\hline
h & 1 & 1 & 0 & 0 & -1 & -1 & -1 & -1 \\
\eta & 0 & 0 & 1 & 1 & 1 & 1 & 1 & 2
\end{array}
\]
and irrelevant ideal $(t_1,t_2)\cap(a_1,a_2,b_1,b_2,b_3,c)$.
Here $\FF$ is a $\PP^5(1^5,2)$-bundle over $\PP^1$, with a section of $\ZZ/2$-quotient singularities. The Chow group of $\FF$ is generated by $h=(1,0)$ and $\eta=(0,1)$, with intersection numbers $h^2=0$, $h\cdot\eta^5=1/2$, $\eta^6=7/4$.

As with no.~5, the toric minimal model program on $\FF$ can be worked out using the secondary fan:
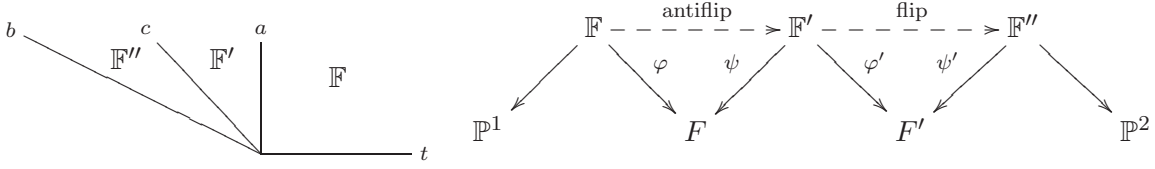
\[\begin{array}{lr}
\xygraph{
!{(0,-1.3) }="A"
!{(0,0) }*+{\scriptstyle{a}}="B"
!{(1.7,-1.3) }*+{\scriptstyle{t}}="C"
!{(-2.6,0) }*+{\scriptstyle{b}}="D"
!{(-1.2,0) }*+{\scriptstyle{c}}="E"
!{(0.8,-0.5) }*+{\FF}
!{(-0.4,-0.3) }*+{\FF'}
!{(-1.4,-0.3) }*+{\FF''}
"A"-"B"  "A"-"C" "A"-"D" "A"-"E"
} 
 &
{\xymatrix{
&\FF\ar[ld]\ar^{\varphi}[rd]\ar^{\textrm{antiflip}}@{-->}[rr]&&{\FF'}\ar[ld]_{\psi}\ar[rd]^{\fie'}\ar^{\textrm{flip}}@{-->}[rr]&&{\FF''}\ar[ld]_{\psi'}\ar[rd]\\
\PP^1&&F&&F'&&\PP^2
}}
\end{array}\]
%

Let $\varphi\colon\FF\to\PP(1^8,2^2)$ be the map induced by taking $\Proj\bigoplus_{m\ge0}H^0(\FF,m\eta)$. The image $F=\fie(\FF)$ is defined by
\[
\bigwedge^2\left(\begin{array}{ccc|c}
x_1 & x_2 & x_3 & y_1  \\
x_4 & x_5 & x_6 & y_2
\end{array}\right)=0
\]
where $x_1=t_1b_1,\dots,x_6=t_2b_3$, $x_7=a_1$, $x_8=a_2$ have weight 1 and $y_1=t_1c$, $y_2=t_2c$ have weight 2. The exceptional locus of $\varphi$ is $\Exc(\fie)\cong\PP^1\times\PP^1$ defined by $b_1=b_2=b_3=c=0$, and this is contracted to the line $\Pi=\PP^1$ with coordinates $x_7,x_8$. The varieties $\Ce$ and $C$ from no.~5 also appear here as $\Ce=\FF\cap(c=0)$ and $C=F\cap(y_1=y_2=0)$. 

Let $\Xhat$ be a codimension 3 complete intersection in $\FF$ of two general elements $f_1,f_2$ in $|2\eta|$, and one general element $g$ in $|{-2h}+3\eta|$. Then $\Xhat$ has degree
\[(-2h+3\eta)\cdot(2\eta)\cdot(2\eta)\cdot \eta^3=-8\cdot\frac12+12\cdot\frac74=17,
\]
and since $K_{\FF}=2h-7\eta$, we see from adjunction that $K_{\Xhat}$ is trivial.
For $i=1,2$, general choices $f_i=t_ic+\dots$ mean that $\Xhat$ is disjoint from the $\ZZ/2$-section of $\FF$. Moreover, in the image of $\fie$, the $f_i$ eliminate the weight 2 variables $y_i$ in favour of quadrics $p_i(x_1,\dots,x_8)$, so that $\fie(\Xhat)$ lies in $\PP^7$.

For dimension reasons, $\Xhat$ is disjoint from $\Exc(\fie)$ and thus $X=\varphi(\Xhat)$ is a smooth Calabi--Yau 3-fold of degree 17 in $\PP^7$ with equations
\begin{equation}\label{D-deg-6}\bigwedge^2\left(\begin{array}{ccc|c}
x_1 & x_2 & x_3 & p_1  \\
x_4 & x_5 & x_6 & p_2
\end{array}\right)=0
\end{equation}
together with
\[
\begin{array}{rclcl}
m_1x_1 + m_2x_2 + m_3x_3 + kp_1 & = & 0 \\
m_1x_4 + m_2x_5 + m_3x_6 + kp_2 & \equiv & n_1x_1 + n_2x_2 + n_3x_3 + lp_1 & = & 0 \\
&& n_1x_4 + n_2x_5 + n_3x_6 + lp_2 & = & 0.
\end{array}\]
As mentioned above, the $p_i$ are general quadrics in $x_1,\dots,x_8$. On the other hand, $m_i,n_i$ are quadrics and $l,k$ are linear forms satisfying the equivalence implied by the second displayed equation, modulo the $2\times2$ minors of the matrix (this is called ``rolling factors''). Thus the minors comprise three quadrics and three cubics, and the rolling factors equations give a further three cubics defining $X$.
We can check using computer algebra that a random example is aG, so a generic $X$ is aG.
\begin{prop}\label{prop!toric-hodge} The Hodge numbers of Calabi--Yau no.~$6$ are $h^{1,1}=2$ and $h^{1,2}=54$.
\end{prop}
\begin{proof}
Let $i\colon X\to \FF$ be the inclusion. There is a short exact sequence
\[0\to T_X\to i^*T_{\FF}\to N_{X/\FF}\to 0\]
so that $c_t(T_X)\cdot c_t(N_{X/\FF})=c_t(i^*T_{\FF})=i^*c_t(T_{\FF})$. Now standard toric methods compute $c_t(T_{\FF})$ to be
\[c_t(T_{\FF})=(1+ht)^2(1+\eta t)^2(1+(-h+\eta)t)^3(1+(-h+2\eta)t),\]
and since $X$ is a complete intersection in $\FF$,
\[c_t(N_{X/\FF})=(1+2\eta t)^2(1+(-2h+3\eta)t).\]
Thus
\[c_t(T_X)=i^*c_t(T_{\FF})(\sum_{k=0}^\infty (-2\eta t)^k)^2\sum_{k=0}^\infty (-(-2h+3\eta)t)^k).\]
After restriction to $X$, the coefficient of $t^3$ is $212h\eta^5 - 120\eta^6 = -104$. Since $h^{11}=2$ by the Lefschetz hyperplane theorem, we have $h^{12}=54$.
\end{proof}

\subsubsection{Other birational models of no.~$6$}
The birational transform of $\Xhat$ can be traced around the toric MMP diagram. 
Here are two more models of $\Xhat$:
\[\Xhat'\colon|2h'|\cap|2h'|\cap|h'+2\eta'|
\subset\FF'
=\begin{array}{ccccccccc}
& t_1 & t_2 & a_1 & a_2 & b_1 & b_2 & b_3 & c \\
\hline
h' & 1 & 1 & 1 & 1 & 0 & 0 & 0 & 1 \\
\eta' & -1 & -1 & 0 & 0 & 1 & 1 & 1 & 1
\end{array}
\]
where the irrelevant ideal of $\FF'$ is $(t_1,t_2,a_1,a_2)\cap(b_1,b_2,b_3,c)$. 
\[\Xhat''\colon|{2h''}-2\eta''|\cap|2h''-2\eta''|\cap|h''+\eta''|
\subset\FF''
=\begin{array}{ccccccccc}
& t_1 & t_2 & a_1 & a_2 & c & b_1 & b_2 & b_3 \\
\hline
h'' & 1 & 1 & 1 & 1 & 1 & 0 & 0 & 0 \\
\eta'' & -2 & -2 & -1 & -1 & 0 & 1 & 1 & 1
\end{array}
\]
where the irrelevant ideal of $\FF''$ is $(b_1,b_2,b_3)\cap(t_1,t_2,a_1,a_2,c)$, so that $\FF''=\PP_{\PP^2}(\oo\oplus2\oo(-1)\oplus2\oo(-2))$, note we moved variable $c$. 

First observe that $\FF\to\PP^1$ induces a fibration on $\Xhat\to\PP^1$ in K3 surfaces of degree six realised as complete intersections $(2,2,3)$ in $\PP^5(1^5,2)$, and $\FF''\to\PP^2$ induces a fibration on $\Xhat''\to\PP^2$ in elliptic curves of degree four as $(1,2,2)$ in $\PP^4$.
Now, the antiflip $\FF\dasharrow\FF'$ contracts $\Exc(\fie)$ and extracts $\Exc(\psi)\cong\PP^1\times\PP^2$ defined by $(t_i=0)$ in $\FF'$. We have shown above that $\fie$ induces an isomorphism $\Xhat\cong X$, and a computation in coordinates shows that $\Exc(\psi)$ is disjoint from $\Xhat'$, so that $\Xhat\cong X\cong\Xhat'$. The flip $\FF'\dasharrow\FF''$ contracts $\PP^3\colon(b_i=0)$ to the point $P_c$, and extracts $\PP^2\colon(a_i=t_j=0)$. The restriction of $\FF'\dasharrow\FF''$ to $\Xhat'$ flops the base locus of $|\eta-h|$, which is a curve. Thus $X$ has a fibration in K3 surfaces of degree 6, and after the flop, $X''$ has an elliptic fibration.

We have proved the following:
 \begin{cor}\label{mov(X)} Calabi--Yau threefold $X$ no.~$6$ has Picard number 2,  the K\"{a}hler cone is generated by the divisors $H_1=h|_X$ and $H_2=2\eta-h|_X$, such that $|H_1|$ induces a K3 fibration and $|H_2|$ a contraction of a line $l$. After flopping $l$, the linear system $|H_2|$ becomes an elliptic fibration induced by $\eta-h$. The K\"{a}hler cones of these two threefolds covers the movable cone $Mov(X)$ such that there are no more birational models of $X$
 (cf.~\cite{Ka}).  
\end{cor}

\subsection{Degree 18}\label{sec!degree-18}
We know two threefolds of degree $6$ in $\PP^7$: $F_2=\PP^1\times\PP^1\times \PP^1$ and a hyperplane section $F_1$ of $\PP^2\times \PP^2$. 

\begin{prop}\label{18} For $i=1,2$, the del Pezzo threefold $F_i$ is bilinked on a complete intersection $Y_{2,2,3}$
to an aG Calabi--Yau threefold $X_i$ of degree $18$ in $\PP^7$.
\end{prop}

\begin{proof}
Fix $i=1$ or $2$. Take three quadrics and one cubic in $\PP^7$, each of which contains the del Pezzo threefold $F_i$. We obtain a threefold $G_i$ which is linked on the complete intersection $S_{2,2,2,3}$ with $F_i$. Now denote by $Y_{2,2,3}$ a generic complete intersection containing $S_{2,2,2,3}$. Since $F_i$ and $S_{2,2,2,3}$ are arithmetically Gorenstein, we infer that $G_i$ is arithmetically Cohen--Macaulay.
We prove as in \cite[Lem.~3.1]{K2}) that $G_i$ is contained in a cubic that does not contain $F_i$. Now let $X_i$ be the threefold linked to $G_i$ through $T_{2,2,3,3}$, the intersection of the cubic with $Y_{2,2,3}$.


We show that $X_i$ is a Calabi--Yau threefold and compute its minimal resolution, following the argument presented in section \ref{deg17no4} based on the mapping cone construction.
The minimal free resolution of the ideal defining $F_i$ is
\[\mathcal{F}\colon0\rightarrow\oo(-6)\rightarrow 9\oo(-4)\rightarrow 16\oo(-3)\rightarrow 9\oo(-2)\rightarrow \mathcal{I}_{F_i}\rightarrow 0\]
for both $i=1$ and $2$. Applying the mapping cone construction and \cite[Exercise 21.23]{E}, we get the following resolution for $\mathcal{I}_{G_i}$:
\[
\mathcal{G}\colon0\rightarrow6\oo(-7)\rightarrow17\oo(-6)\rightarrow 12\oo(-5)\oplus3\oo(-4)\rightarrow 2\oo(-3)\oplus3\oo(-2)\rightarrow \mathcal{I}_{G_i}\rightarrow 0.
\]


Then the mapping cone construction for the second link gives the free resolution:
\begin{align*}
\mathcal{X'}\colon0\rightarrow3\oo(-8)\oplus2\oo(-7)&\rightarrow2\oo(-8)\oplus2\oo(-7)\oplus3\oo(-6)\oplus12\oo(-5)\rightarrow \\
& \oo(-6)\oplus4\oo(-5)\oplus18\oo(-4)\rightarrow8\oo(-3)\oplus2\oo(-2)\rightarrow\mathcal{I}_{X_i}\rightarrow 0.
\end{align*}
Again, using \cite[Exercise 21.23]{E} on $\mathcal{X'}$ and then applying Gorenstein symmetry, we get the following minimal resolution of $\mathcal{I}_X$:
\[0\rightarrow\mathcal{O}(-8)\rightarrow 2\mathcal{O}(-6)\oplus 8\mathcal{O}(-5)\rightarrow 18\mathcal{O}(-4)\rightarrow 8\mathcal{O}(-3)\oplus 2\mathcal{O}(-2)
 \rightarrow\mathcal{I}_{X_i}\rightarrow 0,\]
and we conclude as before, that $K_{X_i}\cong\oo_{X_i}$.

In order to prove the smoothness of $X_1$ and $X_2$, we work out a particular example using the computer, and then we find sufficiently many $4\times4$ minors of the Jacobian matrix of $X_i$, such that the intersection of their simultaneous vanishing locus with $X_i$ is empty. 
\end{proof}
\begin{rem} The two Calabi--Yau threefolds of degree $18$ were obtained before in \cite{GK-jl} and their Hodge numbers $h^{1,2}$ are $46$ and $45$.
This follows from the fact that bilinkage of the cone over a variety is the same as unprojection (this will be made more precise in \cite{CGKK}).

By a parameter count analogous to Section \ref{deg17no4}, we compute the dimensions of the two constructed families to be 46 and 45 respectively. In particular, this shows that a generic member of each family can be obtained by a bilinkage as above. Indeed, the dimension of the family of complete intersections $Y_{2,2,3}$ containing $F_i$ is equal to 
\[\dim\Gr(2,H^0(\mathcal{I}_{F_i}(2)))+\dim H^0(\mathcal{I}_{F_i}(3)|_Q)-\dim\Aut(\PP^7,F_i).\]
Here the first term is 14 and the second term is the dimension of the space of restrictions of cubics containing $F_i$ to a complete intersection $Q$ of two quadrics containing $F_i$, which is 40 ($9\cdot 8-16-1=55$ is the dimension of the space of cubics containing $F_i$, and we subtract 15 for the dimension of the space of cubics containing a complete intersection of two quadrics). The third term is the dimension of the space of automorphisms of $\PP^7$ preserving $F_i$, which is 8 and 9 respectively. This gives $14+40-8=46$ and $14+40-9=45$ parameters respectively. To those we add $h^0(\oo_{Y_{2,2,3}}(F_i+H))-1=8$, the dimension of the system of bilinked manifolds (see \cite{Hartshorne-biliaison}, the bilinkage here is of height 1). Finally, we  subtract 7  as each $X_i$ is contained in a 7-dimensional family of varieties $Y_{2,2,3}$ (the intersection of quadrics is fixed but we have a 7-dimensional choice of cubics).    
\end{rem}

\subsection{Degree 19}\label{sec!degree-19}

Let $F$ be the Segre embedding of $\PP^2\times\PP^2$ in $\PP^8$ defined by $2\times2$ minors of the $3\times3$ matrix
\begin{equation}\label{segre}
\Phi=
\begin{pmatrix}
z_{11} & z_{12} & z_{13} \\ z_{21} & z_{22} & z_{23} \\ z_{31} & z_{32} & z_{33}
\end{pmatrix}.
\end{equation}
Then the degree 6 del Pezzo surface $S_6\subset \PP^6$ is the intersection of $F$ with a linear subspace isomorphic to $\PP^6$.

Let $Y_{13}\subset\PP^6$ be a Calabi--Yau threefold of degree 13 containing $S_6$. Such a $Y$ is defined by Pfaffians of the matrix $M$ below:
\begin{equation}\label{ar}
M=\begin{pmatrix}
A & B & C & D \\
  & z_{31} & z_{21} & z_{22}-z_{33} \\
  & & z_{11} & z_{12} \\
  & & & z_{13}
\end{pmatrix}.
\end{equation}
The row of quadrics $A,B,C,D$ are linear combinations
\[
A = \sum_{i,j}\alpha_{ij}y_{ij},\ B = \sum_{i,j}\beta_{ij}y_{ij},\
C = \sum_{i,j}\gamma_{ij}y_{ij},\ D = \sum_{i,j}\delta_{ij}y_{ij}
\]
where $y_{ij}=\Phi_{ij}$ is the minor of $\Phi$ obtained by deleting the row and column containing $z_{ij}$.
The other entries are chosen so that Pfaffian 1 is contained in the ideal of $S$, but not as a quadric of rank 4: $\Pf_1=(z_{31}z_{13}-z_{11}z_{33})+(z_{11}z_{22}-z_{12}z_{21})$.

\begin{rem} We can also embed $S_6$ as a section of $F_2=\PP^1\times\PP^1\times\PP^1$ in a Pfaffian threefold $Y_{13}$ as above. Indeed we can find a $5\times 5$ matrix of the same shape as (\ref{ar}) with quadrics $A,B,C,D$ in the ideal of $F_2$ such that Pfaffian 1 is also in the ideal of $F_2$.
\end{rem}

The following proposition is a computation, see also Proposition \ref{LB} below:
\begin{prop} The threefold $Y_{13}$ has $28$ nodes as singular points that are contained in $S_6$.
\end{prop}
After blowing up $S_6 \subset Y_{13}$ and flopping the $28$ exceptional curves, we obtain
a Calabi--Yau threefold ${Y'_{13}}$ such that the strict transform of $S_6$ on ${Y'_{13}}$ is isomorphic to $S_6$ and can be contracted to give a singular threefold $X_{19}\subset \PP^7$. 
We say that $X_{19}$ is obtained by unprojecting $S_6\subset Y_{13}$ and we find it is a Calabi--Yau threefold with 13 equations and 24 syzygies (see below for the minimal free resolution of its defining ideal).
In the next section we prove the following main result of this paper:
\begin{thm}\label{main19} The threefold $X_{19}\subset \PP^7$ admits two smoothings by families of smooth Calabi--Yau threefolds $\mathcal{X}_1$ and $\mathcal{X}_2$. Moreover,
$h^{1,1}(X_1)=h^{1,1}(X_2)=2$ and $h^{1,2}(X_1)=h^{1,2}(X_2)+1=37$ where $X_i\in \mathcal{X}_i$
for $i=1,2$.
\end{thm}
\begin{rem} We expect that the K\"{a}hler cone of $X_1$ is generated by two contractions:
one elliptic fibration to $\PP^2$, and the other is a small contraction that maps to the intersection of a quadric and a quartic in $\PP^5$.
\end{rem}

\subsubsection{The proof of Theorem \ref{main19}}
Before proving the theorem, we provide some preliminary results.
Consider the $(2,2)$-Cremona transformation $T\colon \PP^8\dashrightarrow \PP^8$ defined by the linear system of quadrics containing the Segre embedding $F=\PP^2\times \PP^2\subset\PP^8$. It is known that $T$ factors as
$\PP^8\xleftarrow{\tau}P\xrightarrow{\tau'}\PP^8$, where $\tau$ is the blowup of $F$ and $\tau'$ is the contraction of the strict transform of $\Sec(F)$ to $F'\cong\PP^2\times\PP^2$ in the target. The inverse Cremona $T^{-1}$ is thus defined by the linear system of quadrics containing $F'$.

Denote by $\Theta_{13}\supset F$, the codimension $3$ fivefold defined by $4\times 4$ Pfaffians of $M$ treated as a matrix with entries in $\PP^8$. 
Since $\Theta_{13}\subset \PP^8$ is contained in one quadric, the image $T(\Theta_{13})$ is contained in a hyperplane $H$ in $\PP^8$. By the above discussion, the restriction of $T$ to $\Theta_{13}$ factors as follows:
\[ \xymatrix{
&{\Thetahat}\ar[ld]&{\Xihat}\ar@{=}[l]\ar[dr]\\
\PP^8 \supset \Theta_{13}\ar@{-->}[rrr]&&&\Xi_6\subset \PP^7}
\]
where $\Thetahat\to \Theta_{13}$ is the blow-up of $F$ in $\Theta_{13}$, and $\Xihat\to \Xi$ is the blowup of $F'_{H}=F'\cap H$ in $\Xi$.
\begin{lem} The image $T(\Theta_{13})=\Xi_{6}\subset \PP^7$ has degree $6$ and contains the centre $F'_H\subset \Xi_{6}$ of the inverse Cremona transform $T^{-1}|_{H}$. 
\end{lem}
We describe $\Xi\subset \PP^7$ more precisely. Let $y_{ij}$ be the coordinates on $\PP^8$ under the transformation $T$. We used $y_{ij}$ above to denote the $2\times2$ minors of $\Phi$. By \eqref{ar}, $\Xi\subset\PP^7$ is defined by the vanishing of the $3\times3$ minors of
\begin{equation}\label{target-Theta}
N=\begin{pmatrix}
y_{12} & y_{22} & y_{32} & y_{31} \\
y_{13} & y_{23} & y_{33} & y_{21} \\
A & B & C & D
\end{pmatrix}
\end{equation}
in $\PP^8$ intersected with the hyperplane $H\cong\PP^7$ defined by $y_{22}=y_{33}$. Here $A,B,C,D$ are the same expressions as above, but considered as linear forms in coordinates $y_{ij}$. Indeed, we substitute the minors of $\Phi$ for $y_{ij}$. We then observe that the matrix adjoint to the matrix adjoint to $\Phi$ is just $\Phi \cdot \det \Phi$.  Furthermore, working modulo $y_{22}=y_{33}$,  the $2\times 2$ minors of the $2\times 4$ submatrix consisting of the first two rows of $N$ are combinations of minors of $\Phi$, corresponding by the above to the respective $z$-entries in the matrix $M$.  It is then straightforward to show that the $3\times 3$ minors of $N$ are the Pfaffians of $M$ multiplied by $\det\Phi$, working modulo $y_{22}=y_{33}$.

\begin{prop} The image $T(\Theta)$ is a linear section of the secant variety $\Sigma_{2,3}$ of the Segre embedding of $\PP^2\times\PP^3$. Moreover, a general linear section of $\Sigma_{2,3}$ occurs as
the image $T(\Theta)$ for a general $\Theta$. The centre $F'_H$ of $T^{-1}|_H$ is a nodal section $y_{22}=y_{33}$ of $F'=\PP^2\times\PP^2$.
\end{prop}

Recall that $\Sigma_{2,3}\subset\PP^{11}$ is defined by the $3\times3$ minors of a generic $3\times4$ matrix $W$, and is singular along the degree 10 locus isomorphic to $\PP^2\times\PP^3$ where $W$ drops rank. We consider the  small resolution 
\begin{equation}\label{sigma-prime}
\PP^{11}\times\PP^2\supset\Sigma'\to\Sigma_{2,3}
\end{equation}
constructed as $\Sigma'=\left\{(W,v)\in\PP(\mathrm{Mat}_{3\times4})\times\PP^2 : \ker({^tW}) \supset \left< v \right>\right\}$, with generic fibre $\PP^1$ over $\Sing\Sigma$.
Since $\Sigma'$ is  a projective bundle over $\PP^2$, it has Picard number $2$. We infer that $\Sigma_{2,3}$ has Picard rank 1 and Weil divisor class group of rank 2.

A generic five dimensional linear section $\Xi$ of $\Sigma_{2,3}$ is singular along a curve $l$ of degree $10$. Note that the node of $F'_H$ is contained in the curve $l$.
We infer that $\Xi$ has Picard rank $1$ and Weil divisor class group of rank $2$.
Since ${\Xihat}\to \Xi_{6}$ is the blow up of $F_H'$ which intersects $l$ in 14 points, we deduce the following:
\begin{lem}\label{LA} The fivefold $\Xihat$ is singular along a curve $\hat l$ which is the strict transform of $l$.
The rank of the Picard group is $2$ and the group of Weil divisors has rank $3$. 
The variety $\Theta$ is singular along a surface of degree $28$ contained in $F\cong\PP^2\times\PP^2$, and $\Thetahat\to \Theta_{13}$ is a small resolution obtained by blowing up $F\subset \Theta_{13}$ 
\end{lem}
\begin{proof} Since $F'_H\subset \Xi$, it follows that $\Xihat$ is smooth outside the union of $\hat l$ with the singular locus of $F'_H$. We then perform a computation in coordinates to prove that $\Xihat$ has no additional singularities above the singular point of $F_H'$. For the Picard and divisor class group, we observe that $\Xihat$ is a blow up of $\Xi$, and so we add the irreducible exceptional divisor of the blow up to both groups.
\end{proof}

Let $Y_{13}\subset \PP^6$ be the generic codimension $2$ linear section of $\Theta_{13}\subset \PP^8$. 
The threefold $Y_{13}$ contains a del Pezzo surface $S_6$ of degree $6$ being a codimension two section of $\PP^2\times \PP^2$.
\begin{prop}\label{LB} The variety $Y_{13}$ is a nodal Calabi--Yau threefold with Picard group of rank $1$ and Weil divisor class group of rank $3$.
\end{prop}
\begin{proof} Restricting $\Thetahat\to \Theta_{13}$ to $Y_{13}$, we obtain a small resolution 
$\hat{Y}\to Y_{13}$ which is the blow up of $S_6$.
Since $\Thetahat$ is singular along a curve, we obtain that $\hat{Y}_{13}$ is smooth.
In order to see that $Y_{13}$ is nodal we argue as in \cite[Thm.~2.1]{GK-jl}.
By the Ravindra--Srinivas \cite{RS} Lefschetz theorem, we infer that the Picard number of $\hat{Y}$ is $3$. 
\end{proof}
\begin{proof}[Proof of Theorem \ref{main19}]
In order to use the results of Namikawa \cite{N} on existence of smoothings, we have to understand the restrictions of the Weil divisors on $Y_{13}$ to $S_6$, or equivalently of Cartier divisors on $\Yhat_{13}$ to the strict transform of $S_6$.
We have two independent divisors in $\Cl(Y_{13})$; the restriction $L$ of the hyperplane section of $Y_{13}\subset \PP^6$ and the divisor $S_6$ itself. Now $L$ restricts to $S_6$ as an anticanonical divisor and, since $S_6$ is anticanonical, $S_6|_{S_6}$ is the canonical divisor.

We produce a third Weil divisor independent of the above two as follows.
Let $K'$ be the pullback of a hyperplane section from $\PP^2$ to the small resolution $\Sigma\gets\Sigma'\to \PP^2$ of \eqref{sigma-prime}.
We denote by $K$ the image of $K'$ on $\Xi\subset\PP^7$, being a four dimensional subvariety of degree $4$. For example, $K$ is defined by the $2\times2$ minors of the $2\times4$ matrix formed by two general linear combinations of the rows of $N$.

We use Macaulay2 to find the strict transform $K_\Theta$ of $K$ on $\Theta_{13}\subset \PP^8$ through the Cremona transform $T$.
We find that $K_\Theta\subset\PP^8$ has degree $27$ and intersects $F=\PP^2\times \PP^2$
along a divisor in the system $|5H|$, where $H$ is the hyperplane section of $F\subset \PP^8$.

It follows that the group of Cartier divisors on $\Yhat$ restricts
 to the subgroup of divisors on $S_6$ generated by the anticanonical divisor.
 The linear system $H'+S_6$ on $\Yhat$ where $H'$ is the pullback of the hyperplane section from 
 $Y_{13}\subset\PP^6$ defines a primitive contraction of $\Yhat\to X_{19}\subset \PP^7$.
 We are in case $1$ from \cite[Thm.~10]{N} thus $X_{19}$ admits two smoothings. The fact that these smoothings can be performed by non-degenerate aG Calabi--Yau threefolds in $\PP^7$ follows from \cite[Lemma 2.2]{GK-jl}.
 \end{proof}

\subsubsection{Degree 19 by bilinkage}
We can construct special elements of the families $\mathcal{Y}_1$ and $\mathcal{Y}_2$ by Gorenstein bilinkage inside a Pfaffian variety.
As in Section \ref{sec!degree-18}, let $F_1$ be the section of bidegree $(1,1)$ in $\PP^2\times \PP^2$ and $F_2=\PP^1\times\PP^1\times \PP^1$. 
We show that the Calabi--Yau threefold $X_1$ (respectively $X_2$) is bilinked to $F_1$ (resp.~$F_2$) inside the codimension $3$ fourfold $\Pf_{13}\subset\PP^7$ of degree $13$.

\begin{lem} For $i=1,2$ we can embed $F_i$ inside a fourfold $V_i$ of degree $13$ in $\PP^7$ defined by $4\times 4$ Pfaffians of a $5\times 5$ matrix with one row of quadrics. Moreover, $V_1$ (resp.~$V_2$) is singular along a curve of degree $28$ (resp.~$27$).
\end{lem}

Fix $i=1$ or $2$. Let $X^i_{19}\subset V_i$ be the threefold obtained from $F_i$ by a biliaison of height 1 i.e.~we perform two linkages in $V_i$, first $F_i$ by a quadric to obtain $G_i$, and then $G_i$ by a cubic to obtain $X_i$. The fact that $G_i$ is contained in a cubic which does not contain $F_i$ is nontrivial, but this is proved in a similar way to previous cases \cite[Lem.~3.1]{K2}. As in the proof of \cite[Prop.~3.2]{Hartshorne-biliaison}, we infer that $X^1_{19}$ and $X^2_{19}$ are arithmetically Gorenstein. 
The minimal free resolution of the ideal $\mathcal{I}_{X_i}$ is obtained from the mapping cone construction associated to the two linkages as in Section \ref{deg17no4}. In both cases $i=1,2$ we get the following:
$$0\to\oo(-8)\to \oo(-6)\oplus 12\oo(-5)\to 24\oo(-4) \to \oo(-2)\oplus 12\oo(-3)\to \mathcal{I}_{X_i}\to 0.$$
Furthermore we can prove using Macaulay2 (intersecting $X_i$ with an appropriate number of minors of the Jacobian matrix), that the manifolds $X_1$ and $X_2$ obtained by bilinkage are smooth.

It is not clear however, that in this way we construct Calabi--Yau threefolds that are general in their deformation family:
\begin{prob} Find an explicit algebraic description of generic Calabi--Yau threefolds from $\mathcal{Y}_1$ and $\mathcal{Y}_2$. Are the threefolds $X_1$, $X_2$ generic elements of the families $\mathcal{Y}_1$ and $\mathcal{Y}_2$ respectively? 
\end{prob}
\begin{rem} Note that the dimension count is much more subtle in degree 19 than it was in smaller degree cases. This is because the variety $\Pf_{13}$ in which the bilinkage is made is not a complete intersection. Since we lack a proper understanding of the resolutions of $X_1$ and $X_2$, it is difficult to predict the dimension of the family of varieties of type $\Pf_{13}$ containing $X_1$ or $X_2$. 
\end{rem}

\subsection{Degree 20.}
It seems that this case is the most difficult to classify. Only one example is known, defined by $3\times3$ minors of a $4\times4$ matrix of linear forms in $\PP^7$. To see this construction in detail we refer to \cite{KK1}. Recall that the Hodge numbers in this case are $h^{1,1}(X)=2$ and $h^{1,2}(X)=34$. We refer to \cite{JKLM} to see other nice geometric properties of no.~11, including the movable cone.


\end{document}